\documentclass[fleqn]{article}

\usepackage{amsmath,amsfonts,amssymb,amsthm,url,authblk}

\title{Uniform Lyndon interpolation property in propositional modal logics}
\author{Taishi Kurahashi\thanks{kurahashi@n.kisarazu.ac.jp}}
\date{}

\theoremstyle{plain}
\newtheorem{thm}{Theorem}[section]
\newtheorem{lem}[thm]{Lemma}
\newtheorem{prop}[thm]{Proposition}
\newtheorem{cor}[thm]{Corollary}
\newtheorem{fact}[thm]{Fact}
\newtheorem{prob}[thm]{Problem}

\theoremstyle{defn}
\newtheorem{defn}[thm]{Definition}
\newtheorem{ex}[thm]{Example}
\newtheorem{rem}[thm]{Remark}

\newcommand{\K}{{\bf K}}

\newcommand{\Sub}{{\sf Sub}}
\newcommand{\Th}{{\rm Th}}
\newcommand{\Cl}{\mathcal {C}}

\begin{document}
\maketitle

\abstract{
We introduce and investigate the notion of uniform Lyndon interpolation property (ULIP) which is a strengthening of both uniform interpolation property and Lyndon interpolation property. 
We prove several propositional modal logics including ${\bf K}$, ${\bf KB}$, ${\bf GL}$ and ${\bf Grz}$ enjoy ULIP. 
Our proofs are modifications of Visser's proofs of uniform interpolation property using layered bisimulations \cite{Vis96}. 
Also we give a new upper bound on the complexity of uniform interpolants for ${\bf GL}$ and ${\bf Grz}$. 
}

\section{Introduction}\label{Sec:Intro}

Craig's interpolation property was originally proved by Craig \cite{Cra57} for classical first-order predicate logic, and it is a standard property that a logic is expected to possess. 
A lot of investigations of Craig interpolation property have been done in the field of modal logic (see \cite{GM05}). 
A propositional modal logic $L$ has the Craig interpolation property (CIP) if for any formulas $\varphi$ and $\psi$, if $\varphi \to \psi$ is provable in $L$, then there exists a formula $\theta$ containing only propositional variables that occur in both $\varphi$ and $\psi$ such that $\varphi \to \theta$ and $\theta \to \psi$ are provable in $L$. 

Some propositional normal modal logics such as $\K$, ${\bf KD}$, ${\bf KT}$, ${\bf KB}$, ${\bf K4}$, ${\bf S4}$, ${\bf S5}$, ${\bf GL}$ and ${\bf Grz}$ enjoy CIP, and others not (see \cite{Boo80,Gab72,Rau83,Sch76,Smo78}). 
Several weaker versions of interpolation property such as IPD, IPR and WIP are investigated (see \cite{Mak06}). 
On  the other hand, there are two stronger versions of interpolation property, namely Lyndon interpolation property and uniform interpolation property. 

Lyndon's interpolation property was introduced by Lyndon \cite{Lyn59} who proved that classical first order predicate logic enjoys this property. 
A logic $L$ is said to enjoy the Lyndon interpolation property (LIP) if $\varphi \to \psi$ is provable in $L$, then there exists a formula $\theta$ such that $\varphi \to \theta$ and $\theta \to \psi$ are provable in $L$, and the variables occurring in $\theta$ positively (resp.~negatively) occur in both $\varphi$ and $\psi$ positively (resp.~negatively). 
Maskimova \cite{Mak82} and Fitting \cite{Fit83} studied LIP in modal logics, and proved that propositional logics $\K$, ${\bf KD}$, ${\bf KT}$, ${\bf K4}$, ${\bf S4}$ and ${\bf S5}$ possess LIP. 
Maksimova \cite{Mak91} asked whether logics ${\bf GL}$ and ${\bf Grz}$ enjoy LIP, and this problem was recently settled affirmatively for ${\bf GL}$ by Shamkanov \cite{Sham11} and for ${\bf Grz}$ by Maksimova \cite{Mak14}. 
Recently, Kuznets \cite{Kuz16} proved LIP for a wider class of propositional modal logics including the logics in the so-called modal cube of \cite{Gar17}. 
Maksimova \cite{Mak82} showed that there exist normal extensions of ${\bf S5}$ having CIP but do not have LIP (see also \cite{GM05}). 

Pitts \cite{Pit92} proved that intuitionistic propositional logic has the uniform interpolation property. 
A logic $L$ is said to have the uniform interpolation property (UIP) if for any formula $\varphi$ and any finite set $P$ of propositional variables, there exists a formula $\theta$ such that $\theta$ does not contain propositional variables in $P$ and it uniformly interpolates all $L$-provable implications $\varphi \to \psi$ in $L$ where $\psi$ does not contain propositional variables in $P$. 
Shavrukov \cite{Shav93} proved that the propositional modal logic ${\bf GL}$ has UIP. 
UIP for $\K$, ${\bf Grz}$, and ${\bf KT}$ were proved by Ghilardi \cite{Ghi95} and Visser \cite{Vis96}, Visser \cite{Vis96}, and B\'ilkov\'a \cite{Bil07}, respectively. 
See also \cite{Bil16,Iem18}. 
However, it was proved by Ghilardi and Zawadowski \cite{GZ95} that the modal logic ${\bf S4}$ does not enjoy UIP, and B\'ilkov\'a \cite{Bil07} also showed the same result for ${\bf K4}$. 

So far, it has been studied separately that each logic has UIP and that logic has LIP. 
In this paper, we give a framework which can simultaneously derive that a logic enjoys both UIP and LIP. 
Namely, we introduce the notion of uniform Lyndon interpolation property (ULIP), and investigate this newly introduced notion. 

In Section \ref{Sec:IP}, we show that ULIP is actually stronger than both UIP and LIP. 
Also we prove several basic behaviors of ULIP. 
Then we show that ULIP for the propositional modal logics ${\bf K5}$, ${\bf KD5}$, ${\bf K45}$, ${\bf KD45}$, ${\bf KB5}$ and ${\bf S5}$ easily follows from LIP for each of them. 
In Section \ref{Sec:BBS}, we introduce the notion of layered $(P, Q)$-bisimulation between Kripke models which is a main tool of our proofs. 
ULIP for the propositional modal logics ${\bf K}$, ${\bf KD}$, ${\bf KT}$, ${\bf KB}$, ${\bf KDB}$ and ${\bf KTB}$ is proved in Section \ref{Sec:K}. 
Consequently, we obtain both UIP and LIP for these logics. 
UIP for ${\bf KB}$, ${\bf KDB}$ and ${\bf KTB}$ are probably new. 
At last, we prove ULIP for ${\bf GL}$ and ${\bf Grz}$ in Section \ref{Sec:GL}. 
Our proofs of ULIP are modifications of Visser's proofs \cite{Vis96} of UIP using layered bisimulations. 
Especially for ${\bf GL}$ and ${\bf Grz}$, we give a new upper bound on the complexity of uniform interpolants.

\section{Interpolation properties in propositional modal logics}\label{Sec:IP}

In this section, we introduce some variations of interpolation property. 
In particular, we newly introduce the notion of uniform Lyndon interpolation property, and we investigate several basic behaviors of uniform Lyndon interpolation property. 

The language of propositional modal logic consists of countably many propositional variables $p_0, p_1, p_2, \ldots$, the logical constant $\bot$, and the connectives $\to$ and $\Box$. 
The other symbols such as $\top$, $\land$ and $\Diamond$ are introduced as abbreviations. 
Formulas are defined in the usual way. 

\begin{defn}
We define the modal depth $d(\varphi)$ of a formula $\varphi$ recursively as follows: 
\begin{enumerate}
	\item $d(p) = 0$ for each propositional variable $p$; 
	\item $d(\bot) = 0$; 
	\item $d(\varphi \to \psi) = \max\{d(\varphi), d(\psi)\}$; 
	\item $d(\Box \varphi) = d(\varphi) + 1$. 
\end{enumerate}
\end{defn}

For each formula $\varphi$, let $\Sub(\varphi)$ be the set of all subformulas of $\varphi$. 
We recursively define the sets $v^+(\varphi)$ and $v^-(\varphi)$ of variables occurring in $\varphi$ positively and negatively, respectively.   
\begin{enumerate}
	\item $v^+(p_i) = \{p_i\}$ and $v^-(p_i) = \emptyset$; 
	\item $v^+(\bot) = v^-(\bot) = \emptyset$; 
	\item $v^+(\psi \to \theta) = v^-(\psi) \cup v^+(\theta)$ and $v^-(\psi \to \theta) = v^+(\psi) \cup v^-(\theta)$; 
	\item $v^+(\Box \psi) = v^+(\psi)$ and $v^-(\Box \psi) = v^-(\psi)$. 
\end{enumerate}
Let $v(\varphi) = v^+(\varphi) \cup v^-(\varphi)$ be the set of all propositional variables occurring in $\varphi$. 

A set of formulas is said to be a {\it normal logic} if it contains all propositional tautologies and the formula $\Box (p \to q) \to (\Box p \to \Box q)$, and is closed under modus ponens, necessitation and uniform substitution. 
For any normal logic $L$ and any formula $\varphi$, $\varphi \in L$ is also denoted by $L \vdash \varphi$. 
The least normal logic is called $\K$. 
Also for each set $X$ of formulas, the least normal logic including $X$ is denoted by $\K + X$. 
Several normal logics are defined as follows: 

\begin{defn}\leavevmode
\begin{itemize}
	\item ${\bf KD} = \K + \{\neg \Box \bot\}$
	\item ${\bf KT} = \K + \{\Box p \to p\}$
	\item ${\bf K4} = \K + \{\Box p \to \Box \Box p\}$
	\item ${\bf KD4} = \K + \{\neg \Box \bot, \Box p \to \Box \Box p\}$
	\item ${\bf S4} = \K + \{\Box p \to p, \Box p \to \Box \Box p\}$
	\item ${\bf K5} = \K + \{\Diamond p \to \Box \Diamond p\}$
	\item ${\bf KD5} = \K + \{\neg \Box \bot, \Diamond p \to \Box \Diamond p\}$
	\item ${\bf K45} = \K + \{\Box p \to \Box \Box p, \Diamond p \to \Box \Diamond p\}$
	\item ${\bf KD45} = \K + \{\neg \Box \bot, \Box p \to \Box \Box p, \Diamond p \to \Box \Diamond p\}$
	\item ${\bf KB} = \K + \{p \to \Box \Diamond p\}$
	\item ${\bf KDB} = \K + \{\neg \Box \bot, p \to \Box \Diamond p\}$
	\item ${\bf KTB} = \K + \{\Box p \to p, p \to \Box \Diamond p\}$
	\item ${\bf KB5} = \K + \{p \to \Box \Diamond p, \Diamond p \to \Box \Diamond p\}$
	\item ${\bf S5} = \K + \{\Box p \to p, \Diamond p \to \Box \Diamond p\}$
	\item ${\bf GL} = \K + \{\Box (\Box p \to p) \to \Box p\}$
	\item ${\bf Grz} = \K + \{\Box(\Box(p \to \Box p) \to p) \to p\}$
\end{itemize}
\end{defn}

We define the translation $\star$ of formulas as follows (see \cite{Boo80,Gold78}): 
\begin{enumerate}
	\item $p^\star \equiv p$;
	\item $\bot^\star \equiv \bot$; 
	\item $(\varphi \to \psi)^\star \equiv (\varphi^\star \to \psi^\star)$; 
	\item $(\Box \varphi)^\star \equiv \varphi^\star \land \Box \varphi^\star$. 
\end{enumerate}
For any normal logic $L$, let $L^\star$ be the logic $\{\varphi : L \vdash \varphi^\star\}$. 
Then $L^\star$ is also a normal logic. 

\begin{ex}\leavevmode
\begin{itemize}
	\item $\K^\star = {\bf KD}^\star = {\bf KT}$. 
	\item ${\bf KB}^\star = {\bf KDB}^\star = {\bf KTB}$. 
	\item ${\bf K4}^\star = {\bf KD4}^\star = {\bf S4}$. 
	\item ${\bf GL}^\star = {\bf Grz}$ (see \cite{Boo80,Gold78}). 
\end{itemize}
\end{ex}

We introduce the notion of Craig interpolation property (CIP). 
All normal logics introduced above enjoy CIP. 

\begin{defn}
We say a logic $L$ enjoys the {\it Craig interpolation property} (CIP) if for any formulas $\varphi$ and $\psi$, if $L \vdash \varphi \to \psi$, then there exists a formula $\theta$ satisfying the following properties: 
\begin{enumerate}	
	\item $v(\theta) \subseteq v(\varphi) \cap v(\psi)$; 
	\item $L \vdash \varphi \to \theta$; 
	\item $L \vdash \theta \to \psi$. 
\end{enumerate}
Such a formula $\theta$ is said to be a {\it Craig interpolant} of $\varphi \to \psi$ in $L$. 
\end{defn}

Secondly, we introduce Lyndon interpolation property (LIP). 
LIP is stronger than CIP, and all normal logics introduced above also enjoy LIP. 

\begin{defn}
We say a logic $L$ enjoys the {\it Lyndon interpolation property} (LIP) if for any formulas $\varphi$ and $\psi$, if $L \vdash \varphi \to \psi$, then there exists a formula $\theta$ satisfying the following properties: 
\begin{enumerate}	
	\item $v^+(\theta) \subseteq v^+(\varphi) \cap v^+(\psi)$; 
	\item $v^-(\theta) \subseteq v^-(\varphi) \cap v^-(\psi)$; 
	\item $L \vdash \varphi \to \theta$;
	\item $L \vdash \theta \to \psi$. 
\end{enumerate}
Such a formula $\theta$ is said to be a {\it Lyndon interpolant} of $\varphi \to \psi$ in $L$. 
\end{defn}

Thirdly, we introduce uniform interpolation property (UIP). 
UIP is a stronger property than CIP. 

\begin{defn}
We say a logic $L$ enjoys the {\it uniform interpolation property} (UIP) if for any formula $\varphi$ and any finite set $P$ of propositional variables, there exists a formula $\theta$ satisfying the following properties: 
\begin{enumerate}	
	\item $v(\theta) \subseteq v(\varphi) \setminus P$; 
	\item $L \vdash \varphi \to \theta$; 
	\item for all formulas $\psi$, if $v(\psi) \cap P = \emptyset$ and $L \vdash \varphi \to \psi$, then $L \vdash \theta \to \psi$. 
\end{enumerate}
Such a formula $\theta$ is said to be a {\it uniform interpolant} of $(\varphi, P)$ in $L$. 
\end{defn}

At last, we introduce uniform Lyndon interpolation property (ULIP) which is the main subject of this paper. 

\begin{defn}\label{Def:ULIP}
We say a logic $L$ enjoys the {\it uniform Lyndon interpolation property} (ULIP) if for any formula $\varphi$ and any finite sets $P, Q$ of propositional variables, there exists a formula $\theta$ satisfying the following properties: 
\begin{enumerate}
	\item $v^+(\theta) \subseteq v^+(\varphi) \setminus P$; 
	\item $v^-(\theta) \subseteq v^-(\varphi) \setminus Q$; 
	\item $L \vdash \varphi \to \theta$; 
	\item for all formulas $\psi$, if $v^+(\psi) \cap P = v^-(\psi) \cap Q = \emptyset$ and $L \vdash \varphi \to \psi$, then $L \vdash \theta \to \psi$. 
\end{enumerate}
Such a formula $\theta$ is said to be a {\it uniform Lyndon interpolant} of $(\varphi, P, Q)$ in $L$. 
\end{defn}

\begin{rem}
An interpolant $\theta$ defined in Definition \ref{Def:ULIP} is sometimes called a {\it post-interpolant} because it is an interpolant concerning formulas implied by $\varphi$. 
If $L$ enjoys ULIP, then pre-interpolants also exist. 
In fact, for a uniform Lyndon interpolant $\theta$ of $(\neg \varphi, Q, P)$, $\neg \theta$ is a {\it pre-interpolant} of $(\varphi, P, Q)$ in $L$ with respect to ULIP. 
That is,  
\begin{enumerate}
	\item $v^+(\neg \theta) \subseteq v^+(\varphi) \setminus P$; 
	\item $v^-(\neg \theta) \subseteq v^-(\varphi) \setminus Q$; 
	\item $L \vdash \neg \theta \to \varphi$; 
	\item for all formulas $\psi$, if $v^+(\psi) \cap P = v^-(\psi) \cap Q = \emptyset$ and $L \vdash \psi \to \varphi$, then $L \vdash \psi \to \neg \theta$. 
\end{enumerate}

\end{rem}

We show that ULIP is in fact stronger than both UIP and LIP. 

\begin{prop}\label{ULIP}
If a logic $L$ enjoys ULIP, then $L$ also enjoys both UIP and LIP. 
\end{prop}
\begin{proof}
Suppose that $L$ enjoys ULIP. 

(UIP): Let $\varphi$ be any formula and $P$ be any finite set of propositional variables. 
It is easy to see that a uniform Lyndon interpolant of $(\varphi, P, P)$ in $L$ is a uniform interpolant of $(\varphi, P)$ in $L$. 

(LIP): We prove the LIP of $L$. 
Suppose $L \vdash \varphi \to \psi$. 
For $P = v^+(\varphi) \setminus v^+(\psi)$ and $Q = v^-(\varphi) \setminus v^-(\psi)$, let $\theta$ be a uniform Lyndon interpolant of $(\varphi, P, Q)$ in $L$. 
Then $v^+(\theta) \subseteq v^+(\varphi) \setminus P = v^+(\varphi) \cap v^+(\psi)$, $v^-(\theta) \subseteq v^-(\varphi) \setminus Q = v^-(\varphi) \cap v^-(\psi)$ and $L \vdash \varphi \to \theta$. 
Since $v^+(\psi) \cap P = v^-(\psi) \cap Q = \emptyset$, we obtain $L \vdash \theta \to \psi$. 
Therefore $\theta$ is a Lyndon interpolant of $\varphi \to \psi$ in $L$. 
\qed\end{proof}

From this proposition, we can show that a logic $L$ does not have ULIP if $L$ fails to have either UIP or LIP. 
Ghilardi and Zawadowski \cite{GZ95} proved that ${\bf S4}$ does not possess UIP. 
From their result, B\'ilkov\'a \cite{Bil07} derived that ${\bf K4}$ does not have UIP by considering the translation $\star$. 
The following proposition shows a connection between ULIP and the translation $\star$. 

\begin{prop}\label{ST}
Let $L_0$ and $L_1$ be any logics. 
If $L_0 \subseteq L_1 = L_0^\star$ and $L_0$ enjoys ULIP, then $L_1$ also enjoys ULIP. 
\end{prop}
\begin{proof}
Suppose $L_0 \subseteq L_1 = L_0^\star$ and $L_0$ enjoys ULIP. 
Since $L_0 \vdash (\Box p \to p)^\star$, we have $L_1 \vdash \Box p \to p$. 
Then $L_1 \vdash \Box p \leftrightarrow (\Box p)^\star$. 
It follows $L_1 \vdash \varphi \leftrightarrow \varphi^\star$ for all formulas $\varphi$. 

Let $\varphi$ be any formula and $P$, $Q$ be any finite sets of propositional variables. 
Then we obtain a uniform Lyndon interpolant $\theta$ of $(\varphi^\star, P, Q)$ in $L_0$. 
Since $L_0 \subseteq L_1$, $L_1 \vdash \varphi^\star \to \theta$ and hence $L_1 \vdash \varphi \to \theta$. 
Also $v^\circ(\theta) \subseteq v^\circ(\varphi^\star) = v^\circ(\varphi)$ for $\circ \in \{+, -\}$. 
Let $\psi$ be any formula with $L_1 \vdash \varphi \to \psi$ and $v^+(\psi) \cap P = v^-(\psi) \cap Q = \emptyset$. 
Then $L_0 \vdash \varphi^\star \to \psi^\star$. 
By the choice of $\theta$, $L_0 \vdash \theta \to \psi^\star$ because $v^\circ(\psi^\star) = v^\circ(\psi)$ for $\circ \in \{+, -\}$. 
Then $L_1 \vdash \theta \to \psi^\star$ and hence $L_1 \vdash \theta \to \psi$. 
We conclude that $\theta$ is a uniform Lyndon interpolant of $(\varphi, P, Q)$ in $L_1$. 
\qed\end{proof}

\begin{cor}
${\bf K4}$, ${\bf KD4}$ and ${\bf S4}$ do not enjoy ULIP. 
Moreover, if ${\bf K4} \subseteq L \subseteq {\bf S4}$, then $L$ does not enjoy ULIP. 
\end{cor}
\begin{proof}
It can be shown that if ${\bf K4} \subseteq L \subseteq {\bf S4}$, then $L^\star = {\bf S4}$. 
Then this corollary follows from Ghilardi and Zawadowski's result and Propositions \ref{ULIP} and \ref{ST}.
\qed\end{proof}

Next, we show that for logics satisfying the local tabularity, ULIP is nothing but LIP. 

\begin{defn}\label{LT}(See \cite{CZ97})
A logic $L$ is said to be {\it locally tabular} if for any finite set $R$ of propositional variables, there are only finitely many formulas built from variables in $R$ up to $L$-provable equivalence. 
\end{defn}

Of course, every extension of a locally tabular logic is also locally tabular. 

\begin{prop}
If $L$ is locally tabular and enjoys LIP, then $L$ also enjoys ULIP. 
\end{prop}
\begin{proof}
Suppose that $L$ is locally tabular and enjoys LIP. 
Let $\varphi$ be any formula and $P, Q$ be any finite sets of propositional variables. 
For $R = v(\varphi)$, there exists a finite set $S_R$ of formulas built from variables in $R$ such that for all formulas $\psi$ with $v(\psi) \subseteq R$, there exists a formula $\delta \in S_R$ such that $L \vdash \delta \leftrightarrow \psi$. 
In this proof, we temporarily say that a formula $\xi$ is suitable if $v^+(\xi) \subseteq v^+(\varphi) \setminus P$ and $v^-(\xi) \subseteq v^-(\varphi) \setminus Q$. 

Let $\delta_0, \ldots, \delta_{k}$ be all the elements of the finite set
\[
	\{\delta \in S_R : L \vdash \delta \leftrightarrow \xi\ \text{for some suitable formula}\ \xi\ \text{with}\ L \vdash \varphi \to \xi\}. 
\]
For each $i \leq k$, let $\xi_i$ be a suitable formula with $L \vdash \delta_i \leftrightarrow \xi_i$. 
Let 
\[
	\theta \equiv \bigwedge_{i \leq k} \xi_i. 
\]
Then $\theta$ is also suitable and $L \vdash \varphi \to \theta$. 

Let $\psi$ be any formula with $v^+(\psi) \cap P = v^-(\psi) \cap Q = \emptyset$ and $L \vdash \varphi \to \psi$. 
Since $L$ enjoys LIP, we obtain a Lyndon interpolant $\xi$ of $\varphi \to \psi$ in $L$. 
Since $\xi$ is suitable and $L \vdash \varphi \to \xi$, $\xi$ is $L$-equivalent to $\delta_i$ for some $i \leq k$ because of the local tabularity of $L$. 
Then $\xi$ is also $L$-equivalent to $\xi_i$. 
Since $\xi_i$ is a conjunct of $\theta$, we obtain $L \vdash \theta \to \xi$. 
Since $L \vdash \xi \to \psi$, we conclude $L \vdash \theta \to \psi$. 
Therefore $\theta$ is a uniform Lyndon interpolant of $(\varphi, P, Q)$ in $L$. 
We have proved the ULIP of $L$. 
\qed\end{proof}

Nagle and Thomason \cite{NT85} proved that ${\bf K5}$ is locally tabular. 
The logics ${\bf K5}$, ${\bf KD5}$, ${\bf K45}$, ${\bf KD45}$, ${\bf KB5}$ and ${\bf S5}$ are extensions of ${\bf K5}$, and LIP for these logics are proved by Kuznets \cite{Kuz16}. 
Then we obtain the following corollary. 

\begin{cor}
For any extension of ${\bf K5}$, the LIP and ULIP are equivalent. 
In particular, ${\bf K5}$, ${\bf KD5}$, ${\bf K45}$, ${\bf KD45}$, ${\bf KB5}$ and ${\bf S5}$ enjoy ULIP. 
\end{cor}

We say a formula $\varphi$ is {\it constant} if $v(\varphi) = \emptyset$. 
Rautenberg \cite{Rau83} proved that every extension of a modal logic with constant formulas preserves CIP. 
This is also the case for ULIP. 

\begin{prop}\label{CE}
Let $X$ be a set of constant formulas. 
If $L$ enjoys ULIP, then $L + X$ also enjoys ULIP. 
\end{prop}
\begin{proof}
Suppose that $L$ has ULIP. 
Let $\varphi$ be any formula and let $P,Q$ be any finite sets of propositional variables. 
Then we obtain a uniform Lyndon interpolant $\theta$ of $(\varphi, P, Q)$ in $L$. 
We show that $\theta$ is also a uniform Lyndon interpolant of $(\varphi, P, Q)$ in $L + X$. 
Let $\psi$ be any formula with $L + X \vdash \varphi \to \psi$ and $v^+(\psi) \cap P = v^-(\psi) \cap Q = \emptyset$. 
Then by induction on the length of proofs in $L + X$, we can show that there exists a constant formula $\chi$ such that $L + X \vdash \chi$ and $L \vdash \chi \to (\varphi \to \psi)$. 
Since $L \vdash \varphi \to (\chi \to \psi)$ and $v^\circ(\chi \to \psi) = v^\circ(\psi)$ for $\circ \in \{+, -\}$, we obtain $L \vdash \theta \to (\chi \to \psi)$. 
Thus $L + X \vdash \theta \to \psi$. 
\qed\end{proof}

\section{Layered $(P, Q)$-bisimulation}\label{Sec:BBS}

Throughout this section, let $P$ and $Q$ be any finite sets of propositional variables. 
We introduce the notion of layered $(P, Q)$-bisimulation between Kripke models which is a variation of the notion of layered bisimulation in \cite{Vis96} and $n$-bisimulation in \cite{BDV02}. 
We prove some basic facts concerning this notion. 

A tuple $M = (W, \prec, \Vdash)$ is said to be a {\it Kripke model} if $W$ is a non-empty set, $\prec$ is a binary relation on $W$, and $\Vdash$ is a binary relation between $W$ and the set of all formulas satisfying the usual conditions for satisfaction with the following additional condition: $x \Vdash \Box \varphi$ if and only if for all $y \in W$, $y \Vdash \varphi$ if $x \prec y$. 
We say a formula $\varphi$ is {\it valid} in $M$ if $x \Vdash \varphi$ for all $x \in W$. 

\begin{defn}
A formula $\varphi$ is said to be a {\it $(P, Q)$-formula} if $v^+(\varphi) \subseteq P$ and $v^-(\varphi) \subseteq Q$. 
\end{defn}

\begin{prop}
For each $n \in \omega$, there exists a finite set $F_n^{(P, Q)}$ of $(P, Q)$-formulas with modal depth $\leq n$ such that for all $(P, Q)$-formulas $\psi$ with $d(\psi) \leq n$, there exists $\varphi \in F_n^{(P, Q)}$ such that $\K \vdash \varphi \leftrightarrow \psi$. 
\end{prop}
\begin{proof}
This is easily proved by induction on $n$. 
\end{proof}

\begin{defn}\label{Th}
Let $M = (W, \prec, \Vdash)$ be any Kripke model. 
For each $w \in W$ and $n \in \omega$, we define a set $\Th_n^{(P, Q)}(w)$ and a formula $C_n^{(P, Q)}(w)$ as follows: 
\begin{enumerate}
	\item $\Th_n^{(P, Q)}(w) = \{\varphi \in F_n^{(P, Q)} : w \Vdash \varphi\}$. 
	\item $C_n^{(P, Q)}(w) \equiv \bigwedge \Th_n^{(P, Q)}(w)$. 
\end{enumerate}
\end{defn}

\begin{prop}\label{CF}
Let $M = (W, \prec, \Vdash)$ and $M' = (W', \prec', \Vdash')$ be any Kripke models. For any $w \in W$, $w' \in W'$ and $n \in \omega$, the following are equivalent: 
\begin{enumerate}
	\item $\Th_n^{(P, Q)}(w) \subseteq \Th_n^{(P, Q)}(w')$. 
	\item $\Th_n^{(Q, P)}(w') \subseteq \Th_n^{(Q, P)}(w)$. 
	\item $w' \Vdash' C_n^{(P, Q)}(w)$. 
	\item $w \Vdash C_n^{(Q, P)}(w')$. 
\end{enumerate}
\end{prop}
\begin{proof}
The equivalence $(1 \Leftrightarrow 2)$ follows from the fact that $\varphi$ is a $(P, Q)$-formula if and only if $\neg \varphi$ is a $(Q, P)$-formula. 
The equivalences $(1 \Leftrightarrow 3)$ and $(2 \Leftrightarrow 4)$ are direct consequences of Definition \ref{Th}. 
\qed\end{proof}

\begin{defn}
Let $M = (W, \prec, \Vdash)$ and $M' = (W', \prec', \Vdash')$ be any Kripke models. 
We say a relation $Z \subseteq W \times \omega \times W'$ is a {\it layered $(P, Q)$-bisimulation between $M$ and $M'$} if it satisfies the following three conditions: 
\begin{enumerate}
	\item Suppose $(w, n, w') \in Z$. Then
	\begin{itemize}
		\item for any $p \in P$, if $w \Vdash p$, then $w' \Vdash' p$; 
		\item for any $q \in Q$, if $w \nVdash q$, then $w' \nVdash' q$. 
	\end{itemize}
	\item Suppose $(w, n+1, w') \in Z$ and $w \prec x$. 
	Then there exists $x' \in W'$ such that $w' \prec' x'$ and $(x, n, x') \in Z$. 
	\item Suppose $(w, n+1, w') \in Z$ and $w' \prec' x'$. 
	Then there exists $x \in W$ such that $w \prec x$ and $(x, n, x') \in Z$. 
\end{enumerate}
We say a layered $(P, Q)$-bisimulation $Z$ between $M$ and $M'$ is {\it downward closed} if for any $(w, n, w') \in W \times \omega \times W'$, if $(w, n, w') \in Z$, then $(w, m, w') \in Z$ for all $m \leq n$. 
\end{defn}

We prove the main theorem of this section. 

\begin{thm}\label{BBsim}
Let $M = (W, \prec, \Vdash)$ and $M' = (W', \prec', \Vdash')$ be any Kripke models. 
For any $w \in W$, $w' \in W'$ and $n \in \omega$, the following are equivalent: 
\begin{enumerate}
	\item $\Th_n^{(P, Q)}(w) \subseteq \Th_n^{(P, Q)}(w')$. 
	\item There exists a layered $(P, Q)$-bisimulation $Z$ between $M$ and $M'$ such that $(w, n, w') \in Z$. 
	\item There exists a downward closed layered $(P, Q)$-bisimulation $Z$ between $M$ and $M'$ such that $(w, n, w') \in Z$. 
\end{enumerate}
\end{thm}
\begin{proof}
$(3 \Rightarrow 2)$: Obvious. 

$(2 \Rightarrow 1)$: 
We prove by induction on $m$ that for all $m \in \omega$, $x \in W$ and $x' \in W'$, if there exists a layered $(P, Q)$-bisimulation $Z$ between $M$ and $M'$ such that $(x, m, x') \in Z$, then $\Th_m^{(P, Q)}(x) \subseteq \Th_m^{(P, Q)}(x')$. 
Suppose that the statement holds for all $m' < m$, and that there exists a layered $(P, Q)$-bisimulation $Z$ between $M$ and $M'$ such that $(x, m, x') \in Z$. 
We prove by induction on the construction of $\varphi$ that for any formula $\varphi$, 
\begin{enumerate}
	\item if $\varphi$ is a $(P, Q)$-formula, $d(\varphi) \leq m$ and $x \Vdash \varphi$, then $x' \Vdash' \varphi$; 
	\item if $\varphi$ is a $(Q, P)$-formula, $d(\varphi) \leq m$ and $x \nVdash \varphi$, then $x' \nVdash' \varphi$. 
\end{enumerate}
\begin{itemize}
	\item Base Case (i): $\varphi \equiv p$ for some propositional variable $p$. 

	1. If $p$ is a $(P, Q)$-formula and $x \Vdash p$, then $x' \Vdash' p$ because $p \in P$. 

	2. If $p$ is a $(Q, P)$-formula and $x \nVdash p$, then $x' \nVdash' p$ because $p \in Q$.
	
	\item Base Case (ii): $\varphi \equiv \bot$. 
	1 and 2 follow from $x \nVdash \bot$ and $x' \nVdash' \bot$. 

	\item Induction Case (i): $\varphi \equiv (\psi \to \delta)$. 1 and 2 easily follow from induction hypothesis. 

	\item Induction Case (ii): $\varphi \equiv \Box \psi$. 

	1. Suppose $\Box \psi$ is a $(P, Q)$-formula, $d(\Box \psi) \leq m$ and $x' \nVdash' \Box \psi$. 
	Then $\psi$ is also a $(P, Q)$-formula, $d(\psi) \leq m-1$, and there exists $y' \in W'$ such that $x' \prec' y'$ and $y' \nVdash' \psi$. 
	Since $(x, m, x') \in Z$, there exists $y \in W$ such that $x \prec y$ and $(y, m-1, y') \in Z$. 	
	Then $y \nVdash \psi$ by induction hypothesis. 
	Hence $x \nVdash \Box \psi$. 
	
	2. Suppose $\Box \psi$ is $(Q, P)$-formula, $d(\psi) \leq m$ and $x \nVdash \Box \psi$. 
	Then $\psi$ is a $(Q, P)$-formula, $d(\psi) \leq m-1$ and for some $y \in W$, $x \prec y$ and $y \nVdash \psi$. 
	Then there exists $y' \in W'$ such that $x' \prec' y'$ and $(y, m-1, y') \in Z$ because $(x, m, x') \in Z$. 	
	We have $y' \nVdash' \psi$ by induction hypothesis, and hence $x' \nVdash' \Box \psi$. 
\end{itemize}

$(1 \Rightarrow 3)$: 
We prove by induction on $m$ that for all $m \in \omega$, $x \in W$ and $x' \in W'$, if $\Th_m^{(P, Q)}(x) \subseteq \Th_m^{(P, Q)}(x')$, then there exists a downward closed layered $(P, Q)$-bisimulation $Z$ between $M$ and $M'$ such that $(x, m, x') \in Z$. 

\begin{itemize}
	\item Base Case: $m = 0$. Suppose $\Th_0^{(P, Q)}(x) \subseteq \Th_0^{(P, Q)}(x')$. 
	Let $Z = \{(x, 0, x')\}$. 
	
	Suppose $p \in P$ and $x \Vdash p$. 
	Then $p$ is equivalent to a formula in $\Th_0^{(P, Q)}(x)$. 
	Since $\Th_0^{(P, Q)}(x) \subseteq \Th_0^{(P, Q)}(x')$, we have $x' \Vdash' p$. 
	
	Suppose $q \in Q$ and $x \nVdash q$. 
	Then $\neg q$ is equivalent to some formula in $\Th_0^{(P, Q)}(x)$, and hence $x' \nVdash' q$. 
	
	Therefore $Z$ is a downward closed $(P, Q)$-bisimlation between $M$ and $M'$, and $(x, 0, x') \in Z$. 
	
	\item Induction Case: Assume that the statement holds for $m$. 
	Suppose $\Th_{m+1}^{(P, Q)}(x) \subseteq \Th_{m+1}^{(P, Q)}(x')$. 
	
	For each $y \in W$ with $y \succ x$, $y \Vdash C_m^{(P, Q)}(y)$, and hence $x \Vdash \Diamond C_m^{(P, Q)}(y)$. 
	Since $\Diamond C_m^{(P, Q)}(y)$ is equivalent to some formula in $\Th_{m+1}^{(P, Q)}(x)$, we have $x' \Vdash' \Diamond C_m^{(P, Q)}(y)$ because $\Th_{m+1}^{(P, Q)}(x) \subseteq \Th_{m+1}^{(P, Q)}(x')$. 
	Then there exists $y' \in W'$ such that $y' \succ' x'$ and $y' \Vdash' C_m^{(P, Q)}(y)$. 
	By Proposition \ref{CF}, $\Th_m^{(P, Q)}(y) \subseteq \Th_m^{(P, Q)}(y')$. 
	By induction hypothesis, there exists a downward closed layered $(P, Q)$-bisimulation $Z_y$ between $M$ and $M'$ such that $(y, m, y') \in Z_y$. 
	
	In a similar way, we can prove that for each $y' \in W'$ with $y' \succ' x'$, there exist $y \in W$ and a downward closed layered $(P, Q)$-bisimulation $Z_{y'}$ between $M$ and $M'$ such that $y \succ x$ and $(y, m, y') \in Z_{y'}$. 
	
	Let
	\[
		Z = \{(x, k, x') : k \leq m+1\} \cup \bigcup\{Z_y, Z_{y'} : y \succ x, y' \succ' x'\}. 
	\]
	It is easily shown that $Z$ is a downward closed layered $(P, Q)$-bisimulation between $M$ and $M'$, and $(x, m+1, x') \in Z$. 
\end{itemize}
\qed\end{proof}

\section{ULIP for $\K$, ${\bf KD}$, ${\bf KT}$, ${\bf KB}$, ${\bf KDB}$ and ${\bf KTB}$}\label{Sec:K}

In this section, we prove that the logics ${\bf K}$ and ${\bf KB}$ enjoy ULIP. 
As a consequence, we also obtain ULIP for ${\bf KD}$, ${\bf KT}$, ${\bf KDB}$ and ${\bf KTB}$. 
Consequently, we obtain both UIP and LIP for these logics by Proposition \ref{ULIP}. 

Before proving the theorem, we give a Kripke model theoretic characterization of a slightly sharpened version of ULIP. 

\begin{defn}
Let $\Cl$ be a class of Kripke models. 
We say $\Cl$ has ULIP if for any finite sets $P_1, P_2, P_3, Q_1, Q_2$ and $Q_3$ of propositional variables with $P_1$, $P_2$ and $P_3$ are pairwise disjoint and $Q_1$, $Q_2$ and $Q_3$ are pairwise disjoint, any Kriple models $M = (W, \prec, \Vdash)$ and $M' = (W', \prec', \Vdash')$ in $\Cl$, any elements $w \in W$ and $w' \in W'$ and any natural numbers $m, n \in \omega$, if $\Th_n^{(P_2, Q_2)}(w) \subseteq \Th_n^{(P_2, Q_2)}(w')$, then there exists a Kripke model $M^\ast = (W^\ast, \prec^\ast, \Vdash^\ast)$ in $\Cl$ and $w^\ast \in W^\ast$ such that
\begin{enumerate}
	\item $\Th_n^{(P_1 \cup P_2, Q_1 \cup Q_2)}(w) \subseteq \Th_n^{(P_1 \cup P_2, Q_1 \cup Q_2)}(w^\ast)$ and 
	\item $\Th_m^{(P_2 \cup P_3, Q_2 \cup Q_3)}(w^\ast) \subseteq \Th_m^{(P_2 \cup P_3, Q_2 \cup Q_3)}(w')$.  
\end{enumerate}
\end{defn}

\begin{thm}\label{KC}
For any consistent normal modal logic $L$, the following are equivalent: 
\begin{enumerate}
	\item For any formula $\varphi$ and any finite sets $P$, $Q$ of propositional variables, there exists a uniform Lyndon interpolant $\theta$ of $(\varphi, P, Q)$ in $L$ with $d(\theta) \leq d(\varphi)$. 
	\item $L$ is sound and complete with respect to a class $\Cl$ of Kripke models having ULIP. 
\end{enumerate}
\end{thm}
\begin{proof}
$(1 \Rightarrow 2)$: 
Suppose that the condition stated in Clause 1 holds for $L$. 
Let $\Cl$ be a class of all Kripke models in which $L$ is valid. 
Then $L$ is sound and complete with respect to $\Cl$ by the method of the canonical model of $L$ (see \cite{HC96}). 
Let $P_1, P_2, P_3, Q_1, Q_2$ and $Q_3$ be any finite sets of propositional variables with $P_1$, $P_2$ and $P_3$ are pairwise disjoint and $Q_1$, $Q_2$ and $Q_3$ are pairwise disjoint. 
Let $M = (W, \prec, \Vdash)$ and $M' = (W', \prec', \Vdash')$ be any Kripke models in $\Cl$, $w \in W$ and $w' \in W'$ be any elements and $m, n \in \omega$ be any natural numbers. 
Assume $\Th_n^{(P_2, Q_2)}(w) \subseteq \Th_n^{(P_2, Q_2)}(w')$. 

Let $\varphi$ and $\psi$ be the formulas $C_n^{(P_1 \cup P_2, Q_1 \cup Q_2)}(w)$ and $C_m^{(Q_2 \cup Q_3, P_2 \cup P_3)}(w')$, respectively. 
Then we obtain a uniform Lyndon interpolant $\theta$ of $(\varphi, P_1, Q_1)$ in $L$ with $d(\theta) \leq d(\varphi) = n$. 
We have $L \vdash \varphi \to \theta$, $v^+(\theta) \subseteq v^+(\varphi) \setminus P_1 \subseteq P_2$ and $v^-(\theta) \subseteq v^-(\varphi) \setminus Q_1 \subseteq Q_2$. 
Thus $w \Vdash \theta$, and $\theta$ is equivalent to some formula in $\Th_n^{(P_2, Q_2)}(w)$. 
By the assumption, we obtain $w' \Vdash' \theta$. 

Since $w' \nVdash' \neg \psi$, $w' \nVdash' \theta \to \neg \psi$. 
Thus $L \nvdash \theta \to \neg \psi$. 
Hence $L \nvdash \varphi \to \neg \psi$ because $v^+(\neg \psi) \cap P_1 = v^-(\neg \psi) \cap Q_1 = \emptyset$. 
Then there exists a Kripke model $M^\ast = (W^\ast, \prec^\ast, \Vdash^\ast)$ in $\Cl$ and $w^\ast \in W^\ast$ such that $w^\ast \Vdash^\ast \varphi$ and $w^\ast \Vdash^\ast \psi$. 
By Proposition \ref{CF}, we conclude $\Th_n^{(P_1 \cup P_2, Q_1 \cup Q_2)}(w) \subseteq \Th_n^{(P_1 \cup P_2, Q_1 \cup Q_2)}(w^\ast)$ and $\Th_m^{(P_2 \cup P_3, Q_2 \cup Q_3)}(w^\ast) \subseteq \Th_m^{(P_2 \cup P_3, Q_2 \cup Q_3)}(w')$.

$(2 \Rightarrow 1)$: 
Suppose that $L$ is sound and complete with respect to a class $\Cl$ of Kripke models having ULIP. 
Let $\varphi$ be any formula and $P, Q$ be any finite sets of propositional variables. 
Let $P_1 = P$, $P_2 = v^+(\varphi) \setminus P$, $Q_1 = Q$, $Q_2 = v^-(\varphi) \setminus Q$ and $n =d(\varphi)$. 
Also let
\[
	\theta \equiv \bigwedge \{\delta \in F_n^{(P_2, Q_2)} : L \vdash \varphi \to \delta\}.
\]
Then $v^+(\theta) \subseteq v^+(\varphi) \setminus P$, $v^-(\theta) \subseteq v^-(\varphi) \setminus Q$, $L \vdash \varphi \to \theta$ and $d(\theta) \leq n = d(\varphi)$. 
Let $\psi$ be any formula with $v^+(\psi) \cap P = v^-(\psi) \cap Q = \emptyset$ and $L \nvdash \theta \to \psi$. 
We would like to show $L \nvdash \varphi \to \psi$. 

Let $P_3 = v^+(\psi) \setminus v^+(\varphi)$, $Q_3 = v^-(\psi) \setminus v^-(\varphi)$ and $m = d(\psi)$. 
Since $L \nvdash \theta \to \psi$, there exists a Kriple model $M' = (W', \prec', \Vdash')$ in $\Cl$ and $w' \in W'$ such that $w' \Vdash' \theta$ and $w' \nVdash' \psi$. 
Since $w' \nVdash' \theta \to \neg C_n^{(Q_2, P_2)}(w')$, we have $L \nvdash \theta \to \neg C_n^{(Q_2, P_2)}(w')$. 
By the definition of $\theta$, we obtain $L \nvdash \varphi \to \neg C_n^{(Q_2, P_2)}(w')$. 
Then there exists a Kripke model $M = (W, \prec, \Vdash)$ in $\Cl$ and $w \in W$ such that $w \Vdash \varphi$ and $w \Vdash C_n^{(Q_2, P_2)}(w')$. 
By Proposition \ref{CF}, we have $\Th_n^{(P_2, Q_2)}(w) \subseteq \Th_n^{(P_2, Q_2)}(w')$. 
Since $\Cl$ has ULIP, there exists a Kripke model $M^\ast = (W^\ast, \prec^\ast, \Vdash^\ast)$ in $\Cl$ and $w^\ast \in W^\ast$ such that $\Th_n^{(P_1 \cup P_2, Q_1 \cup Q_2)}(w) \subseteq \Th_n^{(P_1 \cup P_2, Q_1 \cup Q_2)}(w^\ast)$ and $\Th_m^{(P_2 \cup P_3, Q_2 \cup Q_3)}(w^\ast) \subseteq \Th_m^{(P_2 \cup P_3, Q_2 \cup Q_3)}(w')$. 

Since $w \Vdash \varphi$ and $\varphi$ is equivalent to a formula in $F_n^{(P_1 \cup P_2, Q_1 \cup Q_2)}$, we have $w^\ast \Vdash^\ast \varphi$. 
Also since $w' \nVdash' \psi$ and $\psi$ is equivalent to a formula in $F_m^{(P_2 \cup P_3, Q_2 \cup Q_3)}$, we have $w^\ast \nVdash^\ast \psi$. 
Hence $w^\ast \nVdash^\ast \varphi \to \psi$. 
We conclude $L \nvdash \varphi \to \psi$. 
\qed\end{proof}

\begin{defn}
The classes of all Kripke models and all symmetric Kripke models are denoted by $\Cl_{\bf K}$ and $\Cl_{\bf B}$, respectively. 
\end{defn}

\begin{fact}(See \cite{HC96})
${\bf K}$ and ${\bf KB}$ are sound and complete with respect to the classes $\Cl_{\bf K}$ and $\Cl_{\bf B}$, respectively. 
\end{fact}

By Theorem \ref{KC}, for ULIP of ${\bf K}$ and ${\bf KB}$, it suffices to prove that the classes $\Cl_{\bf K}$ and $\Cl_{\bf B}$ have ULIP. 
We prove the following lemma by modifying Visser's proof \cite{Vis96}. 

\begin{lem}\label{ML1}
The classes $\Cl_{\bf K}$ and $\Cl_{\bf B}$ have ULIP. 
\end{lem}
\begin{proof}
Let $P_1, P_2, P_3, Q_1, Q_2$ and $Q_3$ be any finite sets of propositional variables with $P_1$, $P_2$ and $P_3$ are pairwise disjoint and $Q_1$, $Q_2$ and $Q_3$ are pairwise disjoint. 
Let $M = (W, \prec, \Vdash)$ and $M' = (W', \prec', \Vdash')$ be any Kripke models, $w \in W$ and $w' \in W'$ be any elements and $m, n \in \omega$ be any natural numbers. 

Suppose $\Th_n^{(P_2, Q_2)}(w) \subseteq \Th_n^{(P_2, Q_2)}(w')$. 
Then there exists a layered $(P_2, Q_2)$-bisimulation $Z$ between $M$ and $M'$ such that $(w, n, w') \in Z$ by Theorem \ref{BBsim}. 

Let $M^+ = (W^+, \prec^+, \Vdash^+)$ be a Kripke model defined as follows:
\begin{enumerate}
	\item $W^+ = W \cup \{\mathbb{I}\}$, where $\mathbb{I}$ is a new object; 
	\item $\prec^+ = \prec \cup \{(x, \mathbb{I}), (\mathbb{I}, x) : x \in W^+\}$; 
	\item for each propositional variable $p$, $w \Vdash^+ p$ if and only if $w \Vdash p$ for $w \in W$, and $\mathbb{I} \nVdash^+ p$. 
\end{enumerate}

It is easy to see that if $M$ is symmetrical, then so is $M^+$.  

Let $\varepsilon$ be a new object and define $0 - 1 = \varepsilon$ and $\varepsilon - 1 = \varepsilon$. 
We define a Kripke model $M^\ast = (W^\ast, \prec^\ast, \Vdash^\ast)$ and an element $w^\ast \in W^\ast$ as follows: 
\begin{enumerate}
	\item $W^\ast = Z \cup \{(\mathbb{I}, \varepsilon, x') : x' \in W'\}$; 
	\item $(x, s, x') \prec^\ast (y, t, y')$ if and only if $x \prec^+ y$, ($t = s - 1$, $t = s$ or $s = t -1$) and $x' \prec' y'$; 
	\item for each propositional variable $p$, $(x, s, x') \Vdash^\ast p$ if and only if one of the conditions from $1$ to $16$ in the following table (Table \ref{C1}) holds: 
	(for instance, Clause 1 in the table expresses the condition `$p \in P_1 \cap Q_1$, $p \notin P_2 \cup P_3 \cup Q_2 \cup Q_3$ and $x \Vdash^+ p$'): 
 
	\begin{table}[h]
	\begin{center}
	\caption{Conditions for the definition of $\Vdash^\ast$}\label{C1}
	\begin{tabular}{|c|c|c|c|c|c|c|l|}
 	\hline
	 & $P_1$ & $P_2$ & $P_3$ & $Q_1$ & $Q_2$ & $Q_3$ & \\
	\hline
	1 & $\checkmark$ &  &  & $\checkmark$ &  &  & $x \Vdash^+ p$ \\
	\hline
	2 & $\checkmark$ &  &  &  & $\checkmark$ &  & $x \Vdash^+ p$ or $x = \mathbb{I}$ \\
	\hline
	3 & $\checkmark$ &  &  &  &  & $\checkmark$ & $x \Vdash^+ p$ or $x' \Vdash' p$ \\
	\hline
	4 & $\checkmark$ &  &  &  &  &  & $x \Vdash^+ p$ \\
	\hline
	5 &  & $\checkmark$ &  & $\checkmark$ &  &  & $x \Vdash^+ p$ \\
	\hline
	6 &  & $\checkmark$ &  &  & $\checkmark$ &  & $x' \Vdash' p$ \\
	\hline
	7 &  & $\checkmark$ &  &  &  & $\checkmark$ & $x' \Vdash' p$ \\
	\hline
	8 &  & $\checkmark$ &  &  &  &  & $x \Vdash^+ p$ \\
	\hline
	9 &  &  & $\checkmark$ & $\checkmark$ &  &  & $x \Vdash^+ p$ and $x' \Vdash' p$ \\
	\hline
	10 &  &  & $\checkmark$ &  & $\checkmark$ &  & $x' \Vdash' p$ \\
	\hline
	11 &  &  & $\checkmark$ &  &  & $\checkmark$ & $x' \Vdash' p$ \\
	\hline
	12 &  &  & $\checkmark$ &  &  &  & $x' \Vdash' p$ \\
	\hline
	13 &  &  &  & $\checkmark$ &  &  & $x \Vdash^+ p$ \\
	\hline
	14 &  &  &  &  & $\checkmark$ &  & $x' \Vdash' p$ \\
	\hline
	15 &  &  &  &  &  & $\checkmark$ & $x' \Vdash' p$ \\
	\hline
	16 &  &  &  &  &  &  & $x \Vdash^+ p$ \\
	\hline
	\end{tabular}
	\end{center}
	\end{table}

	\item $w^\ast = (w, n, w')$. 
\end{enumerate}

	Notice that if both $M^+$ and $M'$ are symmetrical, then $M^\ast$ is also symmetrical. 

	\vspace{0.1in}
	
	{\bf Claim 1}. Suppose $(x, s, x') \in W^\ast$. 
	\begin{enumerate}
		\item If $p \in P_1 \cup P_2$, $x \in W$ and $x \Vdash p$, then $(x, s, x') \Vdash^\ast p$. 
		\item If $p \in Q_1 \cup Q_2$, $x \in W$ and $x \nVdash p$, then $(x, s, x') \nVdash^\ast p$. 
		\item If $p \in P_2 \cup P_3$ and $(x, s, x') \Vdash^\ast p$, then $x' \Vdash' p$. 
		\item If $p \in Q_2 \cup Q_3$ and $(x, s, x') \nVdash^\ast p$, then $x' \nVdash' p$. 
	\end{enumerate}
	\begin{proof}
	1. Suppose $p \in P_1 \cup P_2$, $x \in W$ and $x \Vdash p$. 
	Then $x \Vdash^+ p$. 
	If $p \in P_1$, then one of the conditions $1$, $2$, $3$ and $4$ holds. 
	If not, we have $p \in P_2$. 
	Since $x \in W$, we have $(x, s, x') \in Z$. 
	Since $Z$ is a layered $(P_2, Q_2)$-bisimulation, we obtain $x' \Vdash' p$. 
	Hence one of the conditions $5$, $6$, $7$ and $8$ holds. 
	In either case, we obtain $(x, s, x') \Vdash^* p$.
	
	2. Suppose $p \in Q_1 \cup Q_2$, $x \in W$ and $(x, s, x') \Vdash^* p$.  
	Then one of the conditions $1$, $2$, $5$, $6$, $9$, $10$, $13$ and $14$ holds. 
	If one of the conditions $1$, $2$, $5$, $9$ and $13$ holds, then $x \Vdash^+ p$ because $x \neq \mathbb{I}$. 
	Hence $x \Vdash p$. 
	If one of the conditions $6$, $10$ and $14$ holds, then $x' \Vdash' p$ and $p \in Q_2$. 
	Hence $x \Vdash p$ holds because $(x, s, x') \in Z$ and $Z$ is a layered $(P_2, Q_2)$-bisimulation. 
	
	3. Suppose $p \in P_2 \cup P_3$ and $(x, s, x') \Vdash^\ast p$. 
	Then one of the conditions from $5$ to $12$ holds. 
	If one of the conditions $6$, $7$, $9$, $10$, $11$ and $12$ holds, then $x' \Vdash' p$. 
	If one of the conditions $5$ and $8$ holds, then $x \Vdash^+ p$ and $p \in P_2$. 
	Since $x \Vdash^+ p$ and $\mathbb{I} \nVdash^+ p$, we have $x \neq \mathbb{I}$. 
	Therefore $(x, s, x') \in Z$. 
	We obtain $x' \Vdash' p$ because of $Z$. 
	
	4. Suppose $p \in Q_2 \cup Q_3$ and $x' \Vdash' p$. 
	If $p \notin P_1 \cap Q_2$ or $x = \mathbb{I}$, then one of the conditions $2$, $3$, $6$, $7$, $10$, $11$, $14$ and $15$ holds. 
	If $p \in P_1 \cap Q_2$ and $x \neq \mathbb{I}$, then $(x, s, x') \in Z$ and hence $x \Vdash^+ p$ because of $Z$. 
	In this case, the condition $2$ holds. 
	In either case, we have $(x, s, x') \Vdash^\ast p$. 
	\qed\end{proof}

	\vspace{0.1in}
	
	{\bf Claim 2}. $\Th_n^{(P_1 \cup P_2, Q_1 \cup Q_2)}(w) \subseteq \Th_n^{(P_1 \cup P_2, Q_1 \cup Q_2)}(w^\ast)$. 
	
	\begin{proof}
	Let
	\[
		Z_1 = \{(x, t, (x, s, x')) : (x, s, x') \in Z\ \text{and}\ t \leq s\}.
	\]
	Then $Z_1 \subseteq W \times \omega \times W^\ast$. 
	\begin{enumerate}
		\item Suppose $(x, t, (x, s, x')) \in Z_1$. 
		Then $(x, s, x') \in Z \subseteq W^\ast$. 
		If $p \in P_1 \cup P_2$ and $x \Vdash p$, then $(x, s, x') \Vdash^\ast p$ by Claim 1.1. 
		If $q \in Q_1 \cup Q_2$ and $x \nVdash q$, then $(x, s, x') \nVdash^\ast q$ by Claim 1.2. 
		\item Suppose $(x, t+1, (x, s, x')) \in Z_1$ and $x \prec y$ for $y \in W$. 
		Then $(x, s, x') \in Z$ and $t + 1 \leq s$. 
		Since $s \geq 1$, there exists $y' \in W'$ such that $x' \prec' y'$ and $(y, s-1, y') \in Z$. 
		Then $(x, s, x') \prec^\ast (y, s-1, y')$ and $(y, t, (y, s-1, y')) \in Z_1$ because $t \leq s-1$. 
		\item Suppose $(x, t+1, (x, s, x')) \in Z_1$ and $(x, s, x') \prec^\ast (y, u, y')$ for $(y, u, y') \in W^\ast$. 
		Then $(x, s, x') \in Z$ and $t + 1 \leq s$. 
		By the definition of $\prec^\ast$, either $u = s - 1$, $u = s$, or $s = u -1$.  
		In either case, $t \leq s - 1 \leq u$. 
		Since $u \in \omega$, we have $(y, u, y') \in Z$. 
		Therefore we conclude $x \prec y$ and $(y, t, (y, u, y')) \in Z_1$. 
	\end{enumerate}
	We have proved that $Z_1$ is a layered $(P_1 \cup P_2, Q_1 \cup Q_2)$-bisimulation between $M$ and $M^\ast$. 
	Since $w^\ast = (w, n, w') \in Z$, we have $(w, n, w^\ast) \in Z_1$. 
	By Theorem \ref{BBsim}, we conclude $\Th_n^{(P_1 \cup P_2, Q_1 \cup Q_2)}(w) \subseteq \Th_n^{(P_1 \cup P_2, Q_1 \cup Q_2)}(w^\ast)$. 
	\qed\end{proof}

	\vspace{0.1in}
	
	{\bf Claim 3}. $\Th_m^{(P_2 \cup P_3, Q_2 \cup Q_3)}(w^\ast) \subseteq \Th_m^{(P_2 \cup P_3, Q_2 \cup Q_3)}(w')$. 
	
	\begin{proof}
	Let
	\[
		Z_2 = \{((x, s, x'), t, x') : (x, s, x') \in W^\ast\ \text{and}\ t \in \omega\}. 
	\]
	Then $Z_2 \subseteq W^\ast \times \omega \times W'$. 
	\begin{enumerate}
		\item Suppose $((x, s, x'), t, x') \in Z_2$. 
		Then $(x, s, x') \in W^\ast$. 
		If $p \in P_2 \cup P_3$ and $(x, s, x') \Vdash^\ast p$, then $x' \Vdash' p$ by Claim 1.3. 
		If $q \in Q_2 \cup Q_3$ and $(x, s, x') \nVdash^\ast q$, then $x' \nVdash' q$ by Claim 1.4. 
		\item Suppose $((x, s, x'), t+1, x') \in Z_2$ and $(x, s, x') \prec^\ast (y, u, y')$ for $(y, u, y') \in W^\ast$. 
		Then $x' \prec' y'$ and $((y, u, y'), t, y') \in Z_2$. 
		\item Suppose $((x, s, x'), t + 1, x') \in Z_2$ and $x' \prec' y'$ for $y' \in W'$. 
		If $s \in \{0, \varepsilon\}$, then $(x, s, x') \prec^\ast (\mathbb{I}, \varepsilon, y')$ and $((\mathbb{I}, \varepsilon, y'), t, y') \in Z_2$. 
		If $s \geq 1$, then $(x, s, x') \in Z$ and $x \in W$. 
		Hence there exists $y \in W$ such that $x \prec y$ and $(y, s-1, y') \in Z \subseteq W^\ast$. 
		We have $(x, s, x') \prec^\ast (y, s-1, y')$ and $((y, s-1, y'), t, y') \in Z_2$. 
	\end{enumerate}
	We have proved that $Z_2$ is a layered $(P_2 \cup P_3, Q_2 \cup Q_3)$-bisimulation between $M^\ast$ and $M'$. 
	Since $w^\ast = (w, n, w') \in W^\ast$, we have $(w^\ast, m, w') \in Z_2$. 
	By Theorem \ref{BBsim}, we conclude $\Th_m^{(P_2 \cup P_3, Q_2 \cup Q_3)}(w^\ast) \subseteq \Th_m^{(P_2 \cup P_3, Q_2 \cup Q_3)}(w')$. 
	\qed\end{proof}

We have simultaneously proved that both the classes $\Cl_{\bf K}$ and $\Cl_{\bf KB}$ have ULIP. 
	\qed\end{proof}

\begin{thm}\label{ULIPK}
$\K$ and ${\bf KB}$ enjoy ULIP. 
Moreover, in each of these logics, for any formula $\varphi$ and any finite sets $P$, $Q$ of propositional variables, there exists a uniform Lyndon interpolant $\theta$ of $(\varphi, P, Q)$ with $d(\theta) \leq d(\varphi)$. 
\end{thm}

\begin{cor}
${\bf KD}$, ${\bf KDB}$, ${\bf KT}$ and ${\bf KTB}$ enjoy ULIP. 
Moreover, in each logic $L$ of them, for any formula $\varphi$ and any finite sets $P$, $Q$ of propositional variables, there exists a uniform Lyndon interpolant $\theta$ of $(\varphi, P, Q)$ in $L$ with $d(\theta) \leq d(\varphi)$.
\end{cor}
\begin{proof}
ULIP for ${\bf KD}$ and ${\bf KDB}$ follows from Proposition \ref{CE}. 
Moreover, from the proof of Proposition \ref{CE}, every uniform Lyndon interpolant $\theta$ of $(\varphi, P, Q)$ in ${\bf K}$ (resp.~{\bf KB}) is also a uniform Lyndon interpolant $\theta$ of $(\varphi, P, Q)$ in ${\bf KD}$ (resp.~${\bf KDB}$). 
By Theorem \ref{ULIPK}, $d(\theta) \leq d(\varphi)$ holds. 

ULIP for ${\bf KT}$ and ${\bf KTB}$ follows from Proposition \ref{ST} because ${\bf K}^\star = {\bf KT}$ and ${\bf KB}^\star = {\bf KTB}$. 
Moreover, from the proof of Proposition \ref{ST}, a uniform Lyndon interpolant $\theta$ of $(\varphi, P, Q)$ in ${\bf KT}$ (resp.~${\bf KTB}$) is given as a uniform Lyndon interpolant of $(\varphi^\star, P, Q)$ in ${\bf K}$ (resp.~{\bf KB}). 
It is easy to show that $d(\varphi^\star) = d(\varphi)$. 
Thus $d(\theta) \leq d(\varphi^\star) = d(\varphi)$ by Theorem \ref{ULIPK}. 
\qed\end{proof}

\section{ULIP for ${\bf GL}$ and ${\bf Grz}$}\label{Sec:GL}

In this section, we prove ULIP for ${\bf GL}$ and ${\bf Grz}$. 
For each formula $\varphi$, let $n(\varphi) : = |\{\psi : \Box \psi \in \Sub(\varphi)\}|$. 
Visser \cite{Vis96} proved that for any formula $\varphi$ and any finite set $P$ of propositional variables, there exists a uniform interpolant $\theta$ of $(\varphi, P)$ in ${\bf GL}$ (or ${\bf Grz}$) with $d(\theta) \leq 4n(\varphi) + 1$. 
Our proof of ULIP for ${\bf GL}$ and ${\bf Grz}$ are also based on Visser's proofs, but there are some modifications. 
Then we obtain interpolants in these logics with lower complexity. 
Namely, we prove the existence of uniform Lyndon interpolants $\theta$ with $d(\theta) \leq 3n(\varphi) + 3$. 

First, we prove ULIP for ${\bf GL}$. 
Let $\Cl_{\bf GL}$ be the class of all finite transitive and irreflexive Kripke models. 
It is known that ${\bf GL}$ is sound and complete with respect to the class $\Cl_{\bf GL}$ (see \cite{Boo93}). 

\begin{lem}\label{ML2}
Let $P_1, P_2, P_3, Q_1, Q_2$ and $Q_3$ be any finite sets of propositional variables with $P_1$, $P_2$ and $P_3$ are pairwise disjoint and $Q_1$, $Q_2$ and $Q_3$ are pairwise disjoint, $\varphi$ be any $(P_1 \cup P_2, Q_1 \cup Q_2)$-formula, $M = (W, \prec, \Vdash)$ and $M' = (W', \prec', \Vdash')$ be any Kripke models in $\Cl_{\bf GL}$, $w \in W$ and $w' \in W'$ be any elements, and $m$ be any natural number. 
Suppose $\Th_{3 n(\varphi)+3}^{(P_2, Q_2)}(w) \subseteq \Th_{3 n(\varphi)+3}^{(P_2, Q_2)}(w')$. 
Then there exists a Kripke model $M^\ast = (W^\ast, \prec^\ast, \Vdash^\ast)$ in $\Cl_{\bf GL}$ and $w^\ast \in W^\ast$ such that for any $\psi \in {\sf Sub}(\varphi)$, 
\begin{enumerate}
	\item If $\psi$ is a $(P_1 \cup P_2, Q_1 \cup Q_2)$-formula and $w \Vdash \psi$, then $w^\ast \Vdash^\ast \psi$; 
	\item If $\psi$ is a $(Q_1 \cup Q_2, P_1 \cup P_2)$-formula and $w \nVdash \psi$, then $w^\ast \nVdash^\ast \psi$; 
	\item $\Th_m^{(P_2 \cup P_3, Q_2 \cup Q_3)}(w^\ast) \subseteq \Th_m^{(P_2 \cup P_3, Q_2 \cup Q_3)}(w')$.  
\end{enumerate}

\end{lem}
\begin{proof}
Let $P_1, P_2, P_3, Q_1, Q_2$ and $Q_3$ be any finite sets of propositional variables with $P_1$, $P_2$ and $P_3$ are pairwise disjoint and $Q_1$, $Q_2$ and $Q_3$ are pairwise disjoint. 
Let $\varphi$ be any $(P_1 \cup P_2, Q_1 \cup Q_2)$-formula. 
Let $M = (W, \prec, \Vdash)$ and $M' = (W', \prec', \Vdash')$ be any Kripke models in $\Cl_{\bf GL}$, $w \in W$ and $w' \in W'$ be any elements and $m$ be any natural number. 
Suppose $\Th_{3 n(\varphi) +3}^{(P_2, Q_2)}(w) \subseteq \Th_{3 n(\varphi) + 3}^{(P_2, Q_2)}(w')$. 
Then there exists a downward closed layered $(P_2, Q_2)$-bisimulation $Z$ between $M$ and $M'$ such that $(w, 3n(\varphi)+3, w') \in Z$ by Theorem \ref{CF}. 

We define binary relations $\prec_\varphi$, $\prec_\varphi^s$ and $x \sim_\varphi y$ on $W$ as follows: for $x, y \in W$, 
\begin{itemize}
	\item $x \prec_\varphi y : \Leftrightarrow$ for any $\Box \psi \in \Sub(\varphi)$, if $x \Vdash \Box \psi$, then $y \Vdash \psi \land \Box \psi$;  
	\item $x \prec_\varphi^s y : \Leftrightarrow x \prec_\varphi y$ and for some $\Box \psi \in \Sub(\varphi)$, $x \nVdash \Box \psi$ and $y \Vdash \Box \psi$; 
	\item $x \sim_\varphi y : \Leftrightarrow x = y$ or ($x \prec_\varphi y$ and $y \prec_\varphi x$). 
\end{itemize}
Then $\prec_\varphi$ is transitive, and $\prec_\varphi^s$ is transitive and irreflexive. 
For each $x \in W$, we define the $\varphi$-height $h_\varphi(x)$ of $x$ as follows: $h_\varphi(x) = \sup\{h_\varphi(y)+1 : x \prec_\varphi^s y \in W\}$ (where $\sup \emptyset = 0$). 
By the definition of $\prec_\varphi^s$, there is no $\prec_\varphi^s$-chain of elements of $W$ longer than $n(\varphi) + 1$. 
Thus for all $x \in W$, $h_\varphi(x) \leq n(\varphi)$. 

Notice that if $x \prec_\varphi y \prec_\varphi z$ and $z \not \prec_\varphi y$, then $x \prec_\varphi^s z$. 
Indeed, since $z \not \prec_\varphi y$, $z \Vdash \Box \psi$ and $y \nVdash \psi \land \Box \psi$ for some $\Box \psi \in \Sub(\varphi)$. 
Since $x \prec_\varphi y$, $x \nVdash \Box \psi$. 
By the transitivity of $\prec_\varphi$, we have $x \prec_\varphi z$. 
Therefore we obtain $x \prec_\varphi^s z$. 

Let $\preceq$ and $\preceq'$ be the reflexive closures of $\prec$ and $\prec'$, respectively. 
For $(x, x'), (u, u'), (v, v') \in W \times W'$, we say that $\langle (u, u'), (v, v') \rangle$ is a {\it witness of $(x, x')$} if the following conditions hold:\footnote{Essential parts of the modification of our proof from Visser's are the use of the relation $\prec_\varphi^s$ and this definition of witnesses.} 
\begin{enumerate}
	\item $u \prec v \preceq x$ and $u' \prec' v' \preceq' x'$; 
	\item $x \sim_\varphi v$; 
	\item $(u, 3 h_\varphi(u) + 3, u')$, $(v, 3 h_\varphi(u) + 2, v')$ and $(x, 3 h_\varphi(u) + 1, x')$ are in $Z$. 
\end{enumerate}

We define a Kripke model $M^\ast = (W^\ast, \prec^\ast, \Vdash^\ast)$ and an element $w^\ast \in W^\ast$ as follows:
\begin{enumerate}
	\item $W^\ast = \{(x, x') \in W \times W' : (x, 3h_\varphi(x)+3, x') \in Z$ or $(x, x')$ has a witness$\}$; 
	\item $(x, x') \prec^\ast (y, y')$ if and only if $x \prec_\varphi y$ and $x' \prec' y'$; 
	\item as in the proof of Lemma \ref{ML1}, for each propositional variable $p$, whether $(x, x') \Vdash^\ast p$ or not is defined by referring to a table obtained from Table \ref{C1} by replacing $x \Vdash^+ p$ with $x \Vdash p$ and deleting `or $x = \mathbb{I}$' in Clause 2; 
	\item $w^\ast = (w, w')$. 
\end{enumerate}

Notice that $W^\ast$ is finite because both $W$ and $W'$ are finite. 
The relation $\prec^\ast$ is transitive because so are both $\prec_\varphi$ and $\prec'$. 
Also the irreflexivity of $\prec^\ast$ is inherited from $\prec'$. 
Therefore $M^\ast$ is in $\Cl_{\bf GL}$. 

Since $h_\varphi(w) \leq n(\varphi)$, $3h_\varphi(w) + 3 \leq 3n(\varphi) + 3$. 
Then $(w, 3h_\varphi(w) + 3, w') \in Z$ because $(w, 3n(\varphi) + 3, w') \in Z$ and $Z$ is downward closed. 
Hence $w^\ast = (w, w') \in W^\ast$.

For Clauses 1 and 2 in the statement of the lemma, it suffices to prove the following claim. 

\hspace{0.1in}

	{\bf Claim 1}. 
For any $\psi \in \Sub(\varphi)$ and $(x, x') \in W^\ast$, 
\begin{enumerate}
	\item if $\psi$ is a $(P_1 \cup P_2, Q_1 \cup Q_2)$-formula and $x \Vdash \psi$, then $(x, x') \Vdash^\ast \psi$; 
	\item if $\psi$ is a $(Q_1 \cup Q_2, P_1 \cup P_2)$-formula and $x \nVdash \psi$, then $(x, x') \nVdash^\ast \psi$. 
\end{enumerate}

	\begin{proof}
	We prove 1 and 2 simultaneously for all $(x, x') \in W^\ast$ by induction on the construction of $\psi$.  
	\begin{itemize}
	\item Base Case (i): $\psi \equiv p$ for some propositional variable $p$. 
	Notice that if $(x, x') \in W^\ast$, then $(x, s, x') \in Z$ for some natural number $s$. 
	Then as in the proof of Lemma \ref{ML1}, we can prove that if $p \in P_1 \cup P_2$ and $x \Vdash p$, then $(x, x') \Vdash^\ast p$, and if $q \in Q_1 \cup Q_2$ and $x \nVdash q$, then $(x, x') \nVdash^\ast q$. 
	\item Base Case (ii): $\psi \equiv \bot$. Trivial. 
	\item Induction Case (i): 1 and 2 follow from induction hypothesis. 
	\item Induction Case (ii): $\psi \equiv \Box \delta$. 
	\begin{enumerate}
	\item Suppose $\Box \delta$ is a $(P_1 \cup P_2, Q_1 \cup Q_2)$-formula and $(x, x') \nVdash^\ast \Box \delta$. 
	Then for some $(y, y') \in W^\ast$, $(x, x') \prec^\ast (y, y')$ and $(y, y') \nVdash^\ast \delta$. 
	Since $\delta$ is also a $(P_1 \cup P_2, Q_1 \cup Q_2)$-formula, $y \nVdash \delta$ by induction hypothesis. 
	Since $x \prec_\varphi y$, we obtain $x \nVdash \Box \delta$. 
	\item Suppose $\Box \delta$ is a $(Q_1 \cup Q_2, P_1 \cup P_2)$-formula and $x \nVdash \Box \delta$. 
	We distinguish the following two cases (a) and (b). 
	\begin{itemize}
	\item Case (a): $(x, 3 h_\varphi(x) + 3, x') \in Z$. 
	Since $x \nVdash \Box \delta$, there exists $y \in W$ such that $x \prec y$ and $y \nVdash \delta$. 
	Then there exists $y' \in W'$ such that $x' \prec' y'$ and $(y, 3h_\varphi(x) + 2, y') \in Z$. 
	In this case, $\langle (x, x'), (y, y') \rangle$ is a witness of $(y, y')$ because $(y, 3h_{\varphi}(x) + 1, y') \in Z$. 
	Therefore $(y, y') \in W^\ast$. 
	\item Case (b): $\langle (u, u'), (v, v') \rangle$ is a witness of $(x, x')$. 
	Since the formula $\Box(\Box \delta \to \delta) \to \Box \delta$ is valid in $M$, we have $x \nVdash \Box(\Box \delta \to \delta)$. 
	Then there exists $y \in W$ such that $x \prec y$, $y \Vdash \Box \delta$ and $y \nVdash \delta$. 
	Since $u \prec v \preceq x \prec y$, we have $u \prec y$ and hence $u \prec_\varphi y$. 
	Thus $u \prec_\varphi^s y$ because $u \nVdash \Box \delta$ and $y \Vdash \Box \delta$. 
	It follows that $h_\varphi(y) + 1 \leq h_\varphi(u)$, and $3h_\varphi(y) + 3 \leq 3 h_\varphi(u)$. 

	Since $(x, 3h_\varphi(u) + 1, x') \in Z$, there exists $y' \in W'$ such that $x' \prec' y'$ and $(y, 3h_\varphi(u), y') \in Z$. 
	By the downward closedness of $Z$, we have $(y, 3h_\varphi(y) + 3, y') \in Z$. 
	Therefore $(y, y') \in W^\ast$. 
	\end{itemize}
	In either case, there exists $(y, y') \in W^\ast$ such that $x \prec_\varphi y$, $x' \prec' y'$ and $y \nVdash \delta$. 
	Thus $(x, x') \prec^\ast (y, y')$. 
	Since $\delta$ is a $(Q_1 \cup Q_2, P_1 \cup P_2)$-formula, we obtain $(y, y') \nVdash^\ast \delta$ by induction hypothesis. 
	We conclude $(x, x') \nVdash^\ast \Box \delta$. 
	\end{enumerate}
	\end{itemize}
	\qed\end{proof}

We finish our proof of Lemma \ref{ML2} by proving the following claim which is Clause 3 in the statement. 

\hspace{1in}

	{\bf Claim 2}. $\Th_m^{(P_2 \cup P_3, Q_2 \cup Q_3)}(w^\ast) \subseteq \Th_m^{(P_2 \cup P_3, Q_2 \cup Q_3)}(w')$. 
	
	\begin{proof}
	Let
	\[
		Z_2 = \{((x, x'), t, x') : (x, x') \in W^\ast\ \text{and}\ t \in \omega\}. 
	\]
	Then $Z_2 \subseteq W^\ast \times \omega \times W'$. 
	\begin{enumerate}
		\item Suppose $((x, x'), t, x') \in Z_2$. 
		Then $(x, x') \in W^\ast$. 
		As in the proof of Claim 1, we can prove that if $p \in P_2 \cup P_3$ and $(x, x') \Vdash^\ast p$, then $x' \Vdash' p$, and if $q \in Q_2 \cup Q_3$ and $(x, x') \nVdash^\ast q$, then $x' \nVdash' q$. 
		\item Suppose $((x, x'), t+1, x') \in Z_2$ and $(x, x') \prec^\ast (y, y')$ for $(y, y') \in W^\ast$. 
		Then $x' \prec' y'$ and $((y, y'), t, y') \in Z_2$. 
		\item Suppose $((x, x'), t + 1, x') \in Z_2$ and $x' \prec' y'$ for $y' \in W'$. 
		We distinguish the following two cases (a) and (b):
		\begin{itemize}
		\item Case (a): $(x, 3h_\varphi(x) + 3, x') \in Z$. 
		Then there exists $y \in W$ such that $x \prec y$ and $(y, 3h_\varphi(x) + 2, y') \in Z$. 
		Then $(y, 3h_\varphi(x) + 1, y') \in Z$. 
		Since $\langle (x, x'), (y, y') \rangle$ is a witness of $(y, y')$, we obtain $(y, y') \in W^\ast$. 

		\item Case (b): $\langle (u, u'), (v, v') \rangle$ is a witness of $(x, x')$. 
	Since $v' \preceq' x' \prec' y'$ and $(v, 3h_\varphi(u) + 2, v') \in Z$, there exists $y \in W$ such that $v \prec y$ and $(y, 3h_\varphi(u)+1, y') \in Z$. 
	Since $x \sim_\varphi v$ and $v \prec y$, we have $x \prec_\varphi y$. 
	\begin{itemize}
	\item 	If $y \sim_\varphi v$, then $\langle (u, u'), (v, v') \rangle$ is also a witness of $(y, y')$. 
	\item If $y \not \sim_\varphi v$, then $u \prec_\varphi^s y$ because $u \prec_\varphi v \prec_\varphi y$ and $y \not \prec_\varphi v$. 
	Then $h_\varphi(y) + 1 \leq h_\varphi(u)$, and hence $3h_\varphi(y) + 3 \leq 3h_\varphi(u)$. 
	By the downward closedness of $Z$, $(y, 3h_\varphi(y) + 3, y') \in Z$. 
	\end{itemize}
	In either case, we obtain $(y, y') \in W^\ast$. 
	\end{itemize}
	Hence there exists $(y, y') \in W^\ast$ such that $(x, x') \prec^\ast (y, y')$ and $((y, y'), t, y') \in Z_2$. 
	\end{enumerate}
	We have proved that $Z_2$ is a layered $(P_2 \cup P_3, Q_2 \cup Q_3)$-bisimulation between $M^\ast$ and $M'$. 
	We have $(w^\ast, m, w') \in Z_2$. 
	By Theorem \ref{BBsim}, we conclude $\Th_m^{(P_2 \cup P_3, Q_2 \cup Q_3)}(w^\ast) \subseteq \Th_m^{(P_2 \cup P_3, Q_2 \cup Q_3)}(w')$. 
	\qed\end{proof}
\qed\end{proof}

\begin{thm}\label{ULIPGL}
${\bf GL}$ enjoys ULIP. 
Moreover, there exists a uniform Lyndon interpolant $\theta$ of $(\varphi, P, Q)$ in ${\bf GL}$ with $d(\theta) \leq 3n(\varphi) + 3$ for any formula $\varphi$ and any finite sets $P$, $Q$ of propositional variables. 
\end{thm}
\begin{proof}
This is proved from Lemma \ref{ML2} as in our proof of $(2 \Rightarrow 1)$ of Theorem \ref{KC} by letting
\[
	\theta \equiv \bigwedge \{\delta \in F_{3n(\varphi) + 3}^{(P_0, Q_0)} : L \vdash \varphi \to \delta\}
\]
for $P_0 = v^+(\varphi) \setminus P$ and $Q_0 = v^-(\varphi) \setminus Q$. 
\qed\end{proof}

We prove ULIP for ${\bf Grz}$. 
Let $\Cl_{\bf Grz}$ be the class of all finite transitive and reflexive Kripke models whose irreflexive counterpart is in $\Cl_{\bf GL}$. 
${\bf Grz}$ is sound and complete with respect to the class $\Cl_{\bf Grz}$ (see \cite{Boo93}). 
In this section, we deal with reflexive Kriple models, so we use the symbol $\preceq$ as binary relations of Kripke models. 

Notice that ${\bf Grz}$ proves $\Box (\Box(p \to \Box p) \to p) \to \Box p$ because ${\bf Grz} \vdash \Box \Box (\Box(p \to \Box p) \to p) \to \Box p$ and ${\bf Grz}$ contains ${\bf K4}$ (see van Benthem and Blok \cite{vBB78}). 

\begin{thm}\label{ULIPGrz}
${\bf Grz}$ enjoys ULIP. 
Moreover, there exists a uniform Lyndon interpolant $\theta$ of $(\varphi, P, Q)$ in ${\bf Grz}$ with $d(\theta) \leq 3n(\varphi) + 3$ for any formula $\varphi$ and any finite sets $P$, $Q$ of propositional variables. 
\end{thm}
\begin{proof}
Let $P_1, P_2, P_3, Q_1, Q_2$ and $Q_3$ be any finite sets of propositional variables with $P_1$, $P_2$ and $P_3$ are pairwise disjoint and $Q_1$, $Q_2$ and $Q_3$ are pairwise disjoint. 
Let $\varphi$ be any $(P_1 \cup P_2, Q_1 \cup Q_2)$-formula. 
Let $M = (W, \preceq, \Vdash)$ and $M' = (W', \preceq', \Vdash')$ be any Kripke models in $\Cl_{\bf Grz}$, $w \in W$ and $w' \in W'$ be any elements and $m$ be any natural number. 
Suppose $\Th_{3 n(\varphi) +3}^{(P_2, Q_2)}(w) \subseteq \Th_{3 n(\varphi) + 3}^{(P_2, Q_2)}(w')$, and let $Z$ be a downward closed layered $(P_2, Q_2)$-bisimulation between $M$ and $M'$ such that $(w, 3n(\varphi)+3, w') \in Z$. 
For ULIP of ${\bf Grz}$, it suffices to prove that there exists a Kripke model $M^\ast = (W^\ast, \preceq^\ast, \Vdash^\ast)$ in $\Cl_{\bf Grz}$ and $w^\ast \in W^\ast$ such that for any $\psi \in {\sf Sub}(\varphi)$, 
\begin{enumerate}
	\item If $\psi$ is a $(P_1 \cup P_2, Q_1 \cup Q_2)$-formula and $w \Vdash \psi$, then $w^\ast \Vdash^\ast \psi$; 
	\item If $\psi$ is a $(Q_1 \cup Q_2, P_1 \cup P_2)$-formula and $w \nVdash \psi$, then $w^\ast \nVdash^\ast \psi$; 
	\item $\Th_m^{(P_2 \cup P_3, Q_2 \cup Q_3)}(w^\ast) \subseteq \Th_m^{(P_2 \cup P_3, Q_2 \cup Q_3)}(w')$.  
\end{enumerate}

We define binary relations $\preceq_\varphi$ and $\prec_\varphi^s$ on $W$ as follows: for $x, y \in W$, 
\begin{itemize}
	\item $x \preceq_\varphi y : \Leftrightarrow$ for any $\Box \psi \in \Sub(\varphi)$, if $x \Vdash \Box \psi$, then $y \Vdash \psi \land \Box \psi$;  
	\item $x \prec_\varphi^s y : \Leftrightarrow x \preceq_\varphi y$ and for some $\Box \psi \in \Sub(\varphi)$, $x \nVdash \Box (\psi \to \Box \psi)$ and $y \Vdash \Box (\psi \to \Box \psi)$. 
\end{itemize}
Then $\preceq_\varphi$ is transitive and reflexive because $\preceq$ is reflexive. 
Also $\prec_\varphi^s$ is transitive and irreflexive. 
For each $x \in W$, let $h_\varphi(x)$ be the $\varphi$-height of $x$ with respect to the relation $\prec_\varphi^s$ as in the proof of Lemma \ref{ML2}. 
Then $h_\varphi(x) \leq n(\varphi)$. 

For $(x, x'), (u, u'), (v, v') \in W \times W'$, we say that $\langle (u, u'), (v, v') \rangle$ is a {\it witness of $(x, x')$} if the following conditions hold: 
\begin{enumerate}
	\item $u \preceq v \preceq x$ and $u' \preceq' v' \preceq' x'$; 
	\item $x \preceq_\varphi v$; 
	\item $(u, 3 h_\varphi(u) + 3, u')$, $(v, 3 h_\varphi(u) + 2, v')$ and $(x, 3 h_\varphi(u) + 1, x')$ are in $Z$. 
\end{enumerate}

The definitions of a Kripke model $M^\ast = (W^\ast, \preceq^\ast, \Vdash^\ast)$ and an element $w^\ast \in W^\ast$ are analogous as in the proof of Lemma \ref{ML2}. 
Then $M^\ast$ is in $\Cl_{\bf Grz}$. 
Also we have $w^\ast = (w, w') \in W^\ast$. 

The proof of the clause 3 in the statement is completely analogous as in the proof of Lemma \ref{ML2}. 
It suffices to prove the following claim. 

\hspace{0.1in}

	{\bf Claim 1}. 
For any $\psi \in \Sub(\varphi)$ and $(x, x') \in W^\ast$, 
\begin{enumerate}
	\item if $\psi$ is a $(P_1 \cup P_2, Q_1 \cup Q_2)$-formula and $x \Vdash \psi$, then $(x, x') \Vdash^\ast \psi$; 
	\item if $\psi$ is a $(Q_1 \cup Q_2, P_1 \cup P_2)$-formula and $x \nVdash \psi$, then $(x, x') \nVdash^\ast \psi$. 
\end{enumerate}

	\begin{proof}
	By induction on the construction of $\psi$. 
	We only prove 2 for the case $\psi \equiv \Box \delta$. 
	
	Suppose $\Box \delta$ is a $(Q_1 \cup Q_2, P_1 \cup P_2)$-formula and $x \nVdash \Box \delta$. 
	If $x \Vdash \Box (\delta \to \Box \delta)$, then $x \Vdash \delta \to \Box \delta$, and hence $x \nVdash \delta$. 
	Then $(x, x') \nVdash^\ast \delta$ by induction hypothesis. 
	Since $\preceq^\ast$ is reflexive, $(x, x') \nVdash^\ast \Box \delta$. 
	Thus we may assume $x \nVdash \Box (\delta \to \Box \delta)$.

	We distinguish the following two cases (a) and (b). 
	\begin{itemize}
	\item Case (a): $(x, 3 h_\varphi(x) + 3, x') \in Z$. 
	Since $x \nVdash \Box \delta$, there exists $y \in W$ such that $x \preceq y$ and $y \nVdash \delta$. 
	Then there exists $y' \in W'$ such that $x' \preceq' y'$ and $(y, 3h_\varphi(x) + 2, y') \in Z$. 
	Since $\langle (x, x'), (y, y') \rangle$ is a witness of $(y, y')$, we obtain $(y, y') \in W^\ast$. 
	\item Case (b): $\langle (u, u'), (v, v') \rangle$ is a witness of $(x, x')$. 
	Since the formula $\Box(\Box (\delta \to \Box \delta) \to \delta) \to \Box \delta$ is valid in $M$, we have $x \nVdash \Box(\Box (\delta \to \Box \delta) \to \delta)$. 
	Then there exists $y \in W$ such that $x \preceq y$, $y \Vdash \Box (\delta \to \Box \delta)$ and $y \nVdash \delta$. 
	Since $u \preceq v \preceq x \preceq y$, we have $u \preceq y$ and hence $u \preceq_\varphi y$. 
	Thus $u \prec_\varphi^s y$ because $u \nVdash \Box (\delta \to \Box \delta)$ and $y \Vdash \Box (\delta \to \Box \delta)$. 
	It follows that $h_\varphi(y) + 1 \leq h_\varphi(u)$, and $3h_\varphi(y) + 3 \leq 3 h_\varphi(u)$. 

	Since $(x, 3h_\varphi(u) + 1, x') \in Z$, there exists $y' \in W'$ such that $x' \preceq' y'$ and $(y, 3h_\varphi(u), y') \in Z$. 
	By the downward closedness of $Z$, we have $(y, 3h_\varphi(y) + 3, y') \in Z$. 
	Therefore $(y, y') \in W^\ast$. 
	\end{itemize}
	In either case, there exists $(y, y') \in W^\ast$ such that $x \preceq_\varphi y$, $x' \preceq' y'$ and $y \nVdash \delta$. 
	Since $\delta$ is a $(Q_1 \cup Q_2, P_1 \cup P_2)$-formula, we obtain $(y, y') \nVdash^\ast \delta$ by induction hypothesis. 
	We conclude $(x, x') \nVdash^\ast \Box \delta$ because $(x, x') \preceq^\ast (y, y')$. 
	\qed\end{proof}
This completes our proof of Theorem \ref{ULIPGrz}. 
\qed\end{proof}

We close this paper with the following problems.

\begin{prob}
Is the upper bound $3n(\varphi) + 3$ in the statements of Theorems \ref{ULIPGL} and \ref{ULIPGrz} optimal?
\end{prob}

Let ${\bf Go} = \K + \{\Box(\Box(p \to \Box p) \to p) \land (\Box(p \to \Box p) \to p) \to p\}$. 
It is known that ${\bf Go} \subseteq {\bf GL} \cap {\bf Grz}$ and ${\bf Go}^\star = {\bf Grz}$ (see \cite{Lit07}). 
Then by Proposition \ref{ST}, ULIP of ${\bf Go}$ implies ULIP of ${\bf Grz}$. 
However, ULIP for ${\bf Go}$ is open. 
It is announced in \cite{ZDO15} that ${\bf Go}$ enjoys UIP. 

\begin{prob}
Does ${\bf Go}$ enjoy ULIP?
\end{prob}

The following problem is important for our work, but it is not settled yet. 

\begin{prob}
Is there a logic having both UIP and LIP but does not have ULIP?
\end{prob}

\bibliographystyle{plain}
\bibliography{ref}

\end{document}